\documentclass[preprint,12pt]{elsarticle}

\usepackage{amssymb}
\usepackage{amsmath}
\usepackage{amsthm}
\theoremstyle{definition}
\newtheorem{definition}{Definition}
\theoremstyle{theorem}
\newtheorem{theorem}{Theorem}
\theoremstyle{remark}
\newtheorem{remark}{Remark}
\theoremstyle{theorem}
\newtheorem{lemma}[theorem]{Lemma}     
\newtheorem{corollary}[theorem]{Corollary} 

\usepackage{enumitem}
\journal{}

\begin{document}

\begin{frontmatter}



\title{Regularity properties of distributions of correspondences without countable generation: applications to large games} 


\author{Motoki Otsuka} 
\ead{s2220082@u.tsukuba.ac.jp}
\affiliation{organization={University of Tsukuba},
            addressline={1-1-1}, 
            city={Tennodai},
            postcode={305-8577}, 
            state={Ibaraki},
            country={Japan}}

\begin{abstract}
We show that each of the regularity properties of regular conditional distributions of correspondences—convexity, closedness, compactness, and preservation of closed graphs—is equivalent to the  condition of nowhere equivalence. This result does not require any countable-generation assumptions. As an application, we establish the existence of a pure-strategy equilibrium for large games with general trait spaces. The trait space may be an arbitrary measurable space. As a corollary, we obtain the existence of a pure-strategy equilibrium in semi-anonymous settings in which payoffs depend, in addition to agents’ own actions, on the joint distribution over the space of agents and actions.
\end{abstract}




\end{frontmatter}



\section{Introduction}
\label{sec1}
This paper investigates the regularity properties of the distribution of correspondences—convexity, closedness, compactness, and preservation of closed graphs. It is well known that, on general probability spaces, these properties may fail; the Lebesgue unit interval is a canonical example; see \cite[Example~1]{Sun1996-eo} and the references therein. Accordingly, the literature focuses on structural conditions under which they hold. It is first shown in \cite{Sun1996-eo} that they hold on atomless Loeb probability spaces, and the results extend to the broader class of saturated probability spaces. Moreover, saturated spaces are shown to be not only sufficient but also necessary for each of these regularity properties; see \cite{Keisler2009-ib}. In turn, these regularity properties, together with the saturation property, form the foundation for many results, including the infinite-dimensional extensions of Lyapunov’s convexity theorem and Fatou’s lemma (see \cite{Podczeck2008-zm,Sun2008-pm,Khan2014-ee,Khan2016-oa}), equilibrium existence theorems in general equilibrium theory and game theory (see \cite{Lee2023-yj,Sun2015-yg,Khan2016-lo,Jang2020-dg,Otsuka2024-dt}), and bang–bang/purification principles and subdifferential calculus in optimization (see \cite{Khan2014-ei,Mordukhovich2018-ay}).

More recently, the equivalence result under saturation is extended in \cite{He2018-gj} by separating the $\sigma$-algebra with respect to which a correspondence is measurable from the $\sigma$-algebra with respect to which selections are measurable. It is shown there that nowhere equivalence between the two $\sigma$-algebras is necessary and sufficient for the regularity properties of the distribution of correspondences. The equivalence is further extended there to regular conditional distributions, and this extension is applied in \cite{He2018-gj} to the infinite-dimensional Lyapunov-type theorem for conditional expectations and to equilibrium analysis in mathematical economics. The advantage of this approach is that it applies to standard probability spaces, which are generally not saturated. However, the analysis assumes that the $\sigma$-algebra for the correspondence—and hence the $\sigma$-algebras we condition on—are countably generated; accordingly, the framework in \cite{He2018-gj} is, more or less, restricted to standard probability spaces.

Our main result removes this restriction. While the equivalence between nowhere equivalence and the regularity properties cannot be obtained for distributions without assuming countable generation, we show that for regular conditional distributions of correspondences the same equivalence does hold without any countable-generation assumption. The extension is achieved at essentially no additional cost: the only difference from \cite{He2018-gj} is that, for simplicity, we work with
the closed graph property rather than upper semicontinuity. Moreover, the cardinality condition on the Polish range space is weakened from uncountable to infinite. Accordingly, our results extend those of \cite{He2018-gj} to general (not necessarily standard) probability spaces.\footnote{Technically speaking, both directions (necessity and sufficiency) in the proof of \cite[Theorem 2]{He2018-gj} rely on the countable-generation assumption; obtaining a characterization without it is nontrivial.}

We apply these regularity properties to the equilibrium existence problem in large games with traits. Many economic environments exhibit mass phenomena: each agent’s payoff depends on only aggregate behavior while each individual has negligible impact on the aggregate. These situations are naturally modeled as large games, that is, games with a continuum of players. A notable difference from finite-player models is that, under weaker conditions, pure-strategy equilibria exist in large games. Recently, to incorporate players’ biological or social traits into the analysis of player interdependence, large games with traits are introduced in \cite{Khan2013-pk}, in which a player’s payoff depends on her own action and on the joint distribution of trait-action pairs; in this literature, however, the trait space has typically been restricted to a compact metric or Polish space; see, e.g., \cite{Khan2013-iz, Qiao2014-vw,Qiao2016-tf,He2017-cr,Wu2022-as}.

Building on our main result—the nowhere-equivalence characterization without any countable-generation assumption—we establish the existence of pure-strategy equilibria in large games with general trait spaces. In particular, the trait space may be an arbitrary measurable space. As a corollary, by taking the trait space to be the agent space itself, we obtain pure-strategy equilibria in semi-anonymous settings where each agent’s payoff depends, in addition to her own action, on the joint distribution over the space of agents and actions. We also note that when the trait space is an arbitrary measurable space, the $\sigma$-algebra generated by a trait function—that is, a measurable function assigning a trait to each agent— is generally not countably generated; hence our main result is needed to handle large games beyond Polish trait spaces.

The remainder of the paper is organized as follows. Section~\ref{sec:Basics} introduces basic definitions and a preliminary result. Section~\ref{sec:Regular} presents the main results on nowhere equivalence and regular conditional distributions of correspondences. Section~\ref{sec:Applications} gives applications to large games. Section 5 contains the proof of Theorem~\ref{the:main}.

\section{Basics}
\label{sec:Basics}


Throughout, \((T,\mathcal T,\lambda)\) denotes a complete, atomless probability space and $\mathcal F$ is a sub-$\sigma$-algebra of $\mathcal T$, unless otherwise specified.
Let $X$ be a Polish (completely metrizable separable) space. A correspondence $F:T \rightrightarrows X$ maps $t\in T$ to a nonempty subset $F(t)\subseteq X$; i.e., $F:T\to \mathcal P(X)\setminus\{\emptyset\}$.
$F$ is \emph{closed/compact-valued} if $F(t)$ is closed/compact for $\lambda$-a.e.\ $t$.
$F$ is \emph{measurable} if its graph
    \[
      \operatorname{Gr}(F):=\{(t,x)\in T\times X:\ x\in F(t)\}
    \]
    belongs to the product $\sigma$-algebra $\mathcal T\otimes\mathcal B(X)$, where $\mathcal B(X)$ is the Borel $\sigma$-algebra of $X$.
A mapping $f:T\to X$ is a (measurable) selection of $F$ if $f(t)\in F(t)$ for $\lambda$-almost all $t\in T$.

Let $G:Y\rightrightarrows Z$ with $Y,Z$ Polish spaces.
$G$ is \emph{upper hemicontinuous} at $y_0\in Y$ if for any open set $O_Z\subseteq Z$ with $G(y_0)\subseteq O_Z$, there exists an open neighborhood $O_Y$ of $y_0$ such that for all $y\in O_Y$, $G(y)\subseteq O_Z$.
$G$ is upper hemicontinuous on $Y$ if the above holds at every $y\in Y$. $G$ \emph{has a closed graph} if its graph is a closed subset of $Y \times Z.$

Let $\mathcal F$ be a sub-$\sigma$-algebra of $\mathcal T$, and let $E\in\mathcal F$ with $\lambda(E)>0$.
Define the restricted probability space on $E$ as $(E,\mathcal F^E,\lambda^E)$ where $\mathcal F^E:=\{E\cap E'\,:\,E'\in\mathcal F\}$ and
    \[
      \lambda^E(A):=\frac{\lambda(A)}{\lambda(E)}\qquad (A\in\mathcal F^E).
    \]
Similarly one can define $(E,\mathcal T^E,\lambda^E)$ using the full $\sigma$-algebra $\mathcal T$.

For a measurable space \((S,\mathcal S)\) we write
\(
\mathcal M(S,\mathcal S)
\)
for the set of all probability measures on \((S,\mathcal S)\).
When the underlying $\sigma$-algebra $\mathcal S$ is clear, we simply write $\mathcal M(S)$.
When \(S\) is a Polish space, $\mathcal M(S)$ denotes $\mathcal M\bigl(S,\mathcal B(S)\bigr)$, i.e. the set of Borel probability measures on \(S\).
We equip $\mathcal M(S)$ with the topology of weak convergence.
Let \((S_1,\mathcal S_1)\) and \((S_2,\mathcal S_2)\) be measurable spaces, \(f:S_1\to S_2\) measurable, and \(\mu\in\mathcal M(S_1,\mathcal S_1)\).
We write the pushforward as
\[
\mu f^{-1}\in\mathcal M(S_2,\mathcal S_2),\qquad
(\mu f^{-1})(A):=\mu\bigl(f^{-1}(A)\bigr)\ \text{for }A\in\mathcal S_2.
\]
For \(\mu\in\mathcal M\bigl(S_1\times S_2,\ \mathcal S_1\otimes \mathcal S_2\bigr)\),
let \(\pi_{S_1}:S_1\times S_2\to S_1\) and \(\pi_{S_2}:S_1\times S_2\to S_2\) be the coordinate projections and define the marginals by
\[
\mu|_{S_1} := \mu \pi_{S_1}^{-1}\in \mathcal M(S_1,\mathcal S_1),
\qquad
\mu|_{S_2} := \mu \pi_{S_2}^{-1}\in \mathcal M(S_2,\mathcal S_2).
\]

An $\mathcal F$-measurable \emph{transition probability (stochastic kernel or Young measure)} from $T$ to a Polish space $X$
    is a mapping $\phi:T\to\mathcal M(X)$ such that for every $B\in\mathcal B(X)$,
    the map $t\mapsto\phi(t,B)$ is $\mathcal F$-measurable.
Denote the set of all such kernels by $\mathcal R^{\mathcal F}(X)$ (or simply $\mathcal R^{\mathcal F}$ when clear).
A function $c:T\times X\to\mathbb R$ is $\mathcal F$-\emph{measurable Carath\'eodory} if (i) for each $x\in X$, the map $t\mapsto c(t,x)$ is $\mathcal F$-measurable;
(ii) for each $t\in T$, the map $x\mapsto c(t,x)$ is continuous on $X$.
The weak topology on $\mathcal R^{\mathcal F}(X)$ is the coarsest topology making the functional
    \[
      \phi \ \longmapsto\ \int_T \!\Big[\int_X c(t,x)\,\phi(t,dx)\Big]\,d\lambda(t)
    \]
continuous for every bounded $\mathcal F$-measurable Carath\'eodory function $c$ on $T\times X$.

We now introduce the central notion of \emph{nowhere equivalence} \cite[Definition~1]{He2017-cr}.

\begin{definition}
    $\mathcal T$ is \emph{nowhere equivalent} to $\mathcal F$ if for every
    $D\in\mathcal T$ with $\lambda(D)>0$, there exists a $\mathcal T$-measurable subset
    $D_0\subseteq D$ such that
    \[
      \lambda(D_0 \triangle D_1) > 0 \qquad \text{for all } D_1 \in \mathcal F^{D},
    \]
    where
    $A\triangle B := (A\setminus B)\cup(B\setminus A)$ denotes the symmetric difference.
  \end{definition}

Next, we present two technically convenient conditions—\emph{conditional atomlessness}
\cite[Definition~4.3]{Hoover1984-xr} and the existence of an
\emph{atomless independent supplement} \cite[Definition~2]{He2017-cr}—and
show that these three conditions are equivalent.
We will use this equivalence repeatedly.
\begin{definition}
\begin{enumerate}
    \item The $\sigma$-algebra $\mathcal T$ is \emph{conditionally atomless} over $\mathcal F$ if for every $D\in\mathcal T$ with $\lambda(D)>0$, there exists a $\mathcal T$-measurable subset $D_0\subset D$ such that, on some set of positive probability,

   $$
   0<\lambda(D_0\mid \mathcal F)<\lambda(D\mid \mathcal F).
   $$

   \item
   The $\sigma$-algebra $\mathcal F$ admits an \emph{atomless independent supplement} in $\mathcal T$ if there exists another sub-$\sigma$-algebra $\mathcal H\subset\mathcal T$ such that $(T,\mathcal H,\lambda)$ is atomless, and for any $C_1\in\mathcal F$ and $C_2\in\mathcal H$,

   $$
   \lambda(C_1\cap C_2)=\lambda(C_1)\lambda(C_2).
   $$
\end{enumerate}
\end{definition}

\begin{theorem}\label{the:nowhere}
   Let $(T, \mathcal T, \lambda)$ be an atomless probability space and $\mathcal F$ a sub-$\sigma$-algebra of $\mathcal T.$ The following conditions are equivalent.
\begin{enumerate}[label=(\roman*)]
  \item $\mathcal T$ is nowhere equivalent to $\mathcal F$.
  \item $\mathcal T$ is conditional atomless over $\mathcal F$.
  \item $\mathcal F$ admits an atomless independent supplement in $\mathcal T$.
\end{enumerate}
\end{theorem}

\begin{proof}
Although Lemma~2 in \cite{He2017-cr} establishes this equivalence under the additional hypothesis that $\mathcal F$ is countably generated, the parts proving (i)$\Leftrightarrow$(ii) and (iii)$\Rightarrow$(i) do not use this hypothesis and therefore remain valid in our present setting; see \cite[Footnote~29]{He2017-cr}. Hence it suffices to prove (ii)$\Rightarrow$(iii).
    
Fix the Lebesgue probability space $([0,1],\mathcal B([0,1]),m)$ and consider the
$\mathcal F$-measurable transition probability
\[
K(t,\cdot)=m \qquad (\forall\,t\in T).
\]
Since $\mathcal T$ is conditionally atomless over $\mathcal F$, Lemma 4.4 (iii) in \cite{Hoover1984-xr} yields an $\mathcal T$-measurable
$U:T\to[0,1]$ such that for every Borel $B\subset[0,1]$,
\begin{equation}\label{eq:condlaw}
\lambda(\{U\in B\}\mid\mathcal F)=K(\cdot,B)=m(B)\qquad\text{a.s.}
\end{equation}

\medskip
Let $\mathcal H$ be the $\sigma$-algebra generated by $U$. For any $F\in\mathcal F$ and any Borel $B\subset[0,1]$,
we have
\[
\lambda(\{U\in B\}\cap F)
=\int_F \lambda(\{U\in B\}\mid\mathcal F)\,d\lambda
=\int_F m(B)\,d\lambda
=m(B)\,\lambda(F).
\]
Taking $F=T$ gives $\lambda \bigl( \{U \in B\})=m(B)$, so $\mathcal H$ is independent of $\mathcal F$ and $(T, \mathcal H, \lambda)$ is atomless. Hence, $(T,\mathcal H,\lambda)$ is  an atomless independent supplement of $\mathcal F$.
\end{proof}

\begin{remark}
    The $\sigma$-algebra $\mathcal{T}$ is said to be \emph{relatively saturated} with respect to $\mathcal{F}$ \cite[Definition~2 (ii)]{He2017-cr} if for any Polish spaces $X$ and $Y$, any measure $\mu \in \mathcal{M}(X \times Y)$, and any $\mathcal{F}$-measurable mapping $f$ from $T$ to $X$ with $\mu|_X = \lambda  f^{-1}$, there exists a $\mathcal{T}$-measurable mapping $g$ from $T$ to $Y$ such that $\mu=\lambda  (f,g)^{-1}$.
When $\mathcal{F}$ is countably generated,\footnote{A probability space
is \emph{countably generated} if its $\sigma$-algebra is generated by a countable collection of measurable
sets together with the null sets.} nowhere equivalence is also
equivalent to relative saturation; see \cite[Lemma~2]{He2017-cr}.
 Without the countable generation assumption, nowhere equivalence is, in general, stronger than relative saturation; see \cite[Footnote 29]{He2017-cr}.
\end{remark}

\section{Regular conditional distributions of correspondences}
\label{sec:Regular}

Let $X$ be a Polish space and
$f:T\to X$ be $\mathcal T$-measurable. 
We write $\mu^{\,f\mid\mathcal F}$ for the regular conditional
  distribution (RCD) of $f$ given $\mathcal F$, that is, $\mu^{\,f\mid\mathcal F}:T\times\mathcal B(X)\to[0,1]$ satisfies: (i) for each $t\in T$, $\mu^{\,f\mid\mathcal F}(t,\cdot)$ is a Borel probability measure on $X$; (ii) for every Borel set $B\subseteq X$,
          \[
            \mu^{\,f\mid\mathcal F}(\cdot,B)\;=\;\mathbb{E}\big[\,\mathbf 1_{\{f\in B\}}\ \big|\ \mathcal F\,\big],
          \]
          i.e., it is the conditional expectation of the indicator $1_{\{f\in B\}}$ given $\mathcal F$.\footnote{}
          
Let $F:T\rightrightarrows X$ be an $\mathcal F$-measurable correspondence. Define
\[
D_F^{\mathcal T}
  := \bigl\{\, \lambda  f^{-1} \in \mathcal M(X)
      \;:\; f \text{ is a } \mathcal T\text{-measurable selection of } F \,\bigr\}.
\]
i.e., the set of distributions induced \(\mathcal T\)-measurable selections of \(F\) under \(\lambda\).
In parallel, define
  \[
    \mathcal R_{F}^{(T,\mathcal F)}
      := \big\{\, \mu^{\,f\mid \mathcal F} \in \mathcal R^{\mathcal F} \;:\; f \text{ is a $\mathcal T$-measurable selection of } F \,\big\},
  \]
  i.e., the set of all RCDs induced by
  $\mathcal T$-measurable selections of $F$ given $\mathcal F$.

First, we revisit the characterization of the regularity properties of distributions of correspondences in terms of nowhere equivalence.
\begin{theorem}[\cite{He2018-gj}]
Let $X$ be a fixed uncountable Polish space. Assume that $\mathcal F$ is countably generated.
Then, the condition that $\mathcal{T}$ is nowhere equivalent to $\mathcal{F}$ is necessary and sufficient for each of the following properties.

\begin{description}
  \item[\textbf{A1}] For any closed valued $\mathcal{F}$-measurable correspondence $F:T\rightrightarrows X$, $D_{F}^{\mathcal{T}}$ is convex.

  \item[\textbf{A2}] For any closed valued $\mathcal{F}$-measurable correspondence $F:T\rightrightarrows X$, $D_{F}^{\mathcal{T}}$ is closed.

  \item[\textbf{A3}] For any compact valued $\mathcal{F}$-measurable correspondence $F:T\rightrightarrows X$, $D_{F}^{\mathcal{T}}$ is compact.

  \item[\textbf{A4}] 
  Let $Y$ be a Polish space and let $G:T\times Y \rightrightarrows X$ be a closed-valued correspondence such that:
\begin{enumerate}
\item for every $y\in Y$, $G(\cdot,y)$ (denoted as $G_y$) is $\mathcal F$-measurable from $T$ to $X$;
\item for every $t\in T$, $G(t,\cdot)$ is upper semicontinuous from $Y$ to $X$;
\item there exists an $\mathcal F$-measurable, compact-valued correspondence $F:T\rightrightarrows X$ with $G(t,y)\subseteq F(t)$ for all $(t,y)\in T\times Y$.
\end{enumerate}
Then the correspondence $H:Y\rightrightarrows \mathcal M(X)$ defined by $H(y):=D^{\mathcal T}_{G_y}$ is upper semicontinuous.
  \item[\textbf{A5}] For any $\mathcal F$-measurable mapping $G:T\to\mathcal M(X)$, there exists a $\mathcal T$-measurable mapping $f:T\to X$ such that \begin{enumerate}
  \item for every Borel subset $B\subset X$,
\[
\lambda  f^{-1}(B)=\int_T G(t)(B)\,d\lambda(t);
\]
  \item $f(t)\in \operatorname{supp} G(t)$ for $\lambda$-almost all $t\in T$, where $\operatorname{supp} G(t)$ denotes the support of the probability measure $G(t)$ on $X$ (the smallest closed subset of $X$ with full $G(t)$-measure).
\end{enumerate}
\end{description}
\end{theorem}

In this characterization, the assumption that $\mathcal{F}$ is countably generated cannot be dropped. Indeed, if $(T,\mathcal{T},\lambda)$ is a saturated probability space and we take $\mathcal{F}=\mathcal{T}$, then all of A1--A5 hold \cite[Theorem~3.6]{Keisler2009-ib}, whereas $\mathcal{T}$ is not nowhere equivalent to $\mathcal{F}$.
By contrast, the regularity properties of regular conditional distributions can be characterized by nowhere equivalence without assuming countable generation. This is a main result of this paper and corresponds to \cite[Theorem~2]{He2017-cr}. Note that the Polish space is assumed to be merely infinite.

\begin{theorem}\label{the:main}
Let $X$ be a fixed infinite Polish space. The following conditions, (NE) and (B1)-(B5), are equivalent.

  \begin{description}
    \item[NE] The $\sigma$-algebra $\mathcal T$ is nowhere equivalent to $\mathcal F$.
    \item[B1] For any sub-$\sigma$-algebra $\mathcal G\subseteq \mathcal F$ and any closed-valued $\mathcal F$-measurable $F:T\rightrightarrows X$,
          $\mathcal R_F^{(T,\mathcal G)}$ is \emph{convex}.
    \item[B2] For any sub-$\sigma$-algebra $\mathcal G\subseteq \mathcal F$ and any closed-valued $\mathcal F$-measurable $F:T\rightrightarrows X$,
          $\mathcal R_F^{(T,\mathcal G)}$ is \emph{weakly closed}.
    \item[B3] For any sub-$\sigma$-algebra $\mathcal G\subseteq \mathcal F$ and any compact-valued $\mathcal F$-measurable $F:T\rightrightarrows X$,
          $\mathcal R_F^{(T,\mathcal G)}$ is \emph{weakly compact}.
    \item[B4]
    Let $\mathcal G$ be a sub-$\sigma$-algebra of $\mathcal F$ and let $Z$ be a metric space and let $G:T\times Z \rightrightarrows X$ be a closed-valued correspondence such that:
    \begin{enumerate}
    \item for every $z\in Z$, $G(\cdot,z)$ (denoted as $G_z$) is $\mathcal F$-measurable from $T$ to $X$;
    \item for every $t\in T$, $G(t,\cdot)$ has a closed graph;
    \item there exists an $\mathcal F$-measurable, compact-valued correspondence $F:T\rightrightarrows X$ with $G(t,z)\subseteq F(t)$ for all $(t,z)\in T\times Z$.
    \end{enumerate}
    Then the correspondence $H:Z\rightrightarrows \mathcal R^\mathcal G(X)$ defined by $H(z):=\mathcal R^{(\mathcal T, \mathcal G)}_{G_z}$ has a closed graph.
  \item[B5] For any sub-$\sigma$-algebra $\mathcal G\subseteq \mathcal F$ and any $G\in \mathcal R^{\mathcal G}$, there exists a $\mathcal T$-measurable mapping $g:T\to X$ such that
\[
\mu^{g\mid\mathcal G}=G.
\]
  \end{description}
\end{theorem}

\begin{proof}
    See Section \ref{sec:Proof}.
\end{proof}

Using Theorem~\ref{the:main}, we obtain a saturation-type characterization of nowhere equivalence without countable generation. This is used in Section \ref{sec:Applications}.
\begin{theorem}\label{the:saturation}
Let $(T,\mathcal T,\lambda)$ be an atomless probability space and let $\mathcal F\subset\mathcal T$ be a sub-$\sigma$-algebra. The following conditions are equivalent:
\begin{enumerate}[]
  \item $\mathcal F$ is nowhere equivalent to $\mathcal T$ with respect to $\lambda$.
  \item For any measurable space $(S,\mathcal S)$, any Polish space $Y$, any $\mu \in \mathcal{M}(S \times Y)$, and any $\mathcal F$-measurable mapping $f:T\to S$ with \(\mu|_S = \lambda f^{-1}\),
  there exists a $\mathcal T$-measurable mapping $g:T\to Y$ such that
  \[
    \mu=\lambda (f,g)^{-1}.
  \]
\end{enumerate}
\end{theorem}

\begin{proof}
1 $\Rightarrow$ 2: Since $Y$ is Polish, from the disintegration theorem \cite[Corollary~A5]{Valadier1990-wx} (see also \cite[Theorem~3.4]{Kallenberg1997-dl}) there exists an $\mathcal S$-measurable transition probability $\kappa:S\times\mathcal B(Y)\to[0,1]$ such that
  \[
    \mu(E\times B)=\int_E \kappa(s,B)\,\mu|_S(ds)\qquad(E\in\mathcal S,\ B\in\mathcal B(Y)).
  \]
Define $K:T\times\mathcal B(Y)\to[0,1]$ by
  \[
    K(t,B):=\kappa\bigl(f(t),B\bigr).
  \]
  Then $K$ is $\mathcal F$-measurable transition probability.

Since the nowhere-equivalence condition is equivalent to B5 in Theorem~\ref{the:main}, there exists a $\mathcal T$-measurable $g:T\to Y$ such that
  \[
    \mu^{g|\mathcal F} = K.
  \]
For rectangles $E\times B$ with $E\in\mathcal S$ and $B\in\mathcal B(Y)$,
  \[
  \begin{aligned}
  \mu(E\times B)
  &= \int_E \kappa(s,B)\,\mu|_S(ds)\\
  &= \int_{f^{-1}(E)} \kappa\!\left(f(t),B\right) d\lambda \qquad\text{(change of variables using $\mu|_S=\lambda f^{-1}$)} \\
  &= \int_{f^{-1}(E)}\,K(t,B)\,d\lambda\\
  &= \lambda(f^{-1}(E) \cap g^{-1}(B)) \qquad\text{(since $\mu^{g|\mathcal F} = K$)}\\
  &= \lambda  (f,g)^{-1}(E \times B).
  \end{aligned}
  \]
  By the $\pi$–$\lambda$ theorem \cite[Theorem~3.2]{Billingsley1986-wb}, this extends to all of $\mathcal S\otimes\mathcal B(Y)$, so $\mu=\lambda(f,g)^{-1}$.

\noindent
2 $\Rightarrow$ 1: Apply the condition 2 with
\[
(S,\mathcal S)=(T,\mathcal F),\qquad Y=[0,1],\qquad \mu=\lambda\otimes m,\qquad f=\mathrm{id}_T,
\]
where $m$ denotes Lebesgue measure on $[0,1]$ and $\mathrm{id}_T$ the identity map on $T$. Then there exists a $\mathcal T$-measurable $g:T\to[0,1]$ such that
\[
\lambda\otimes m=\lambda(\mathrm{id}_T,g)^{-1}.
\]

Let $\mathcal H$ be the $\sigma$-algebra generated by $g$. For any $F\in\mathcal F$ and any Borel $B\subset[0,1]$,
\[
\lambda\bigl(F\cap g^{-1}(B)\bigr)
=\lambda\bigl((\mathrm{id}_T,g)^{-1}(F\times B)\bigr)
=(\lambda\otimes m)(F\times B)
=\lambda(F)\,m(B).
\]
Taking $F=T$ gives $\lambda\bigl(g^{-1}(B)\bigr)=m(B)$, so $\mathcal H$ is independent of $\mathcal F$ and $(T, \mathcal H, \lambda)$ is atomless. Hence, by Theorem~\ref{the:nowhere}, $\mathcal T$ is nowhere equivalent to $\mathcal F$ with respect to $\lambda$.

\end{proof}

\section{Applications to large games}
\label{sec:Applications}

A large game with traits is a continuum-player game in which each player’s payoff depends on their own action and on the joint distribution over trait–action pairs.
Formally, the agent space is a complete, atomless probability space $(T,\mathcal T,\lambda)$. 
The common action space $A$ is a compact metric space. The trait space is an arbitrary measurable space $(S,\mathcal S)$. Fix a sub-$\sigma$-algebra $\mathcal F\subseteq\mathcal T$. 
A \emph{trait function} is an $\mathcal F$-measurable function $\alpha:T\to S$ (each player is assigned a trait from the trait space). Endow $\mathcal M(S\times A)$ with the coarsest topology making the functional $\rho\mapsto\int_{S\times A}\!\varphi(s,a)\,d\rho$ continuous for every bounded Carathéodory function $\varphi$ and with its induced Borel $\sigma$-algebra. 
Define $\mathcal V_{(A,S)}:=C_b\bigl(A\times \mathcal M(S\times A)\bigr)$, the space of bounded continuous functions on $A\times \mathcal M(S\times A)$, equipped with the sup-norm topology and the induced Borel $\sigma$-algebra. 
We refer to an $\mathcal F$-measurable function $v:T\to\mathcal V_{(A,S)}$ as a \emph{game function}; since $\mathcal V_{(A,S)}$ is generally nonseparable, we assume that $v$ is strongly $\mathcal F$-measurable.\footnote{Here “strongly $\mathcal F$-measurable” means that $v:T \to \mathcal V_{(A,S)}$ is, almost everywhere, the pointwise limit of $\mathcal F$-measurable countably-valued functions.}
For each $t\in T$, write $v_t:=v(t)\in\mathcal V_{(A,S)}$, so $v_t:A\times \mathcal M(S\times A)\to\mathbb R$ is bounded and continuous. We call the pair $G=(\alpha,v)$ an $\mathcal F$-\emph{measurable large game with traits}.

\begin{definition}
    A $\mathcal T$-measurable function $f:T\to A$ is a \emph{pure-strategy Nash equilibrium}
    for the $\mathcal F$-measurable large game with traits $G=(\alpha,v)$
    if, for $\lambda$-a.e.\ $t\in T$,
    \[
      v_t\big(f(t),\,\lambda(\alpha,f)^{-1}\big)\ \ge\
      v_t\big(a,\,\lambda(\alpha,f)^{-1}\big)\quad \text{for all } a\in A .
    \]
\end{definition}

\begin{remark}
    When $S$ is Polish, the weak topology on $\mathcal M(S \times A)$ is coarser than the topology introduced above. Consequently, the class of payoff functions considered here is broader than the class in the previous literature. In particular, the framework of large game with traits in \cite[Section 3]{He2017-cr} is a special case of the present one.
\end{remark}

In general, large games may not have a Nash equilibrium; this is demonstrated in \cite{Rath1995-so} by an example in which the agent space is the Lebesgue unit interval. The same is true for large games with traits: if $\mathcal T = \mathcal F$, there are some examples with no Nash equilibrium even when the trait space is compact metric; see \cite[Example~1]{Khan2013-pk} and \cite[Example~1]{Wu2022-as}.
The following theorem shows that the nowhere-equivalence condition between $\mathcal T$ and $\mathcal F$ is sufficient for the existence of pure-strategy Nash equilibria in large games with an arbitrary measurable trait space.
\begin{theorem}\label{the:Nash}
    Assume that the $\sigma$-algebra $\mathcal T$ is nowhere equivalent to the sub-$\sigma$-algebra $\mathcal F$. Then, every $\mathcal F$-measurable large game with traits $G=(\alpha,v)$ has a $\mathcal T$-measurable Nash equilibrium.
\end{theorem}

The main difficulty in the proof of Theorem~\ref{the:Nash} is that $\mathcal{R}^{\mathcal{F}}(A)$ is, in general, not metrizable. We therefore adopt the distributional game approach in \cite{Mas-Colell1984-xg}. To that end, we prepare a lemma, which is a slight extension of Portmanteau theorem.

\begin{lemma}\label{lem:Port}
Let \(X\) be a Polish space and \(Z\) a topological space. For any \(A\subset X\times Z\) and \(w\in Z\), write
\[
A_w:=\{x\in X:\ (x,w)\in A\}.
\]
Let \((\mu_\alpha)_\alpha\) be a net of probability measures on \(X\) with \(\mu_\alpha\Rightarrow\mu\) (weak convergence), and let \((z_\alpha)_\alpha\) be a net in \(Z\) with \(z_\alpha\to z\).
\begin{enumerate}
\item[\textnormal{(i)}] If \(G\subset X\times Z\) is open, then
\[
\liminf_{\alpha}\mu_\alpha\!\left(G_{z_\alpha}\right)\ \ge\ \mu\!\left(G_{z}\right).
\]
\item[\textnormal{(ii)}] If \(F\subset X\times Z\) is closed, then
\[
\limsup_{\alpha}\mu_\alpha\!\left(F_{z_\alpha}\right)\ \le\ \mu\!\left(F_{z}\right).
\]
\end{enumerate}
\end{lemma}

\begin{proof}
It suffices to prove (i), since (ii) follows by applying (i) to complements.
Fix \(\varepsilon>0\). Because \(X\) is Polish, there exists a compact set \(K\subset G_z\) with \(\mu(K)\ge \mu(G_z)-\varepsilon.\)
For each \(x\in K\), we have \((x,z)\in G\). Since \(G\) is open, there exist an open neighborhood \(U_x\subset X\) of \(x\) and an open neighborhood \(V_x\subset Z\) of \(z\) such that
\[
U_x\times V_x\subset G .
\]
Since \(\{U_x\}_{x\in K}\) is an open cover of the compact set \(K\), choose \(x_1,\dots,x_m\in K\) with
\[
K\subset O:=\bigcup_{i=1}^m U_{x_i}.
\]
Let
\[
V:=\bigcap_{i=1}^m V_{x_i}.
\]
Then \(V\) is an open neighborhood of \(z\) and
\(
O\times V\subset G .
\)

Because \(z_\alpha\to z\), there exists \(\alpha_0\) such that \(z_\alpha\in V\) for all \(\alpha\ge \alpha_0\). In particular, we have \(O\times\{z_\alpha\}\subset G\), hence
\[
O\subset G_{z_\alpha}\qquad(\alpha\ge \alpha_0).
\]
By monotonicity of measures,
\[
\mu_\alpha(G_{z_\alpha})\ \ge\ \mu_\alpha(O)\qquad(\alpha\ge \alpha_0).
\]
Therefore
\[
\liminf_{\alpha}\mu_\alpha(G_{z_\alpha})\ \ge\ \liminf_{\alpha}\mu_\alpha(O).
\]
Since \(O\) is open, by Portmanteau theorem,
\[
\liminf_{\alpha}\mu_\alpha(O)\ \ge\ \mu(O).
\]
Combining the inequalities, we have
\[
\liminf_{\alpha}\mu_\alpha(G_{z_\alpha})
\ \ge\ \liminf_{\alpha}\mu_\alpha(O)
\ \ge\ \mu(O)
\ \ge\ \mu(K)
\ \ge\ \mu(G_z)-\varepsilon .
\]
Letting \(\varepsilon\downarrow0\) yields the conclusion.
\end{proof}

\begin{proof}[Proof of Theorem~\ref{the:Nash}]
Consider an $\mathcal F$-measurable large game with traits $G=(\alpha,v)$. Since $v:T\to\mathcal V(A,S)$
is strongly $\mathcal F$-measurable: it is essentially separably valued.\footnote{That is, there exists a null set $N \in \mathcal F$ such that $v (T \setminus N)$ is separable.} Fix a separable closed subset
$V \subset \mathcal V(A,S)$ with $v(t)\in V$ for $\lambda$-a.e.\ $t$.
Without loss of generality, we may assume  $v(t)\in V$ for all $t$.
Equip $\mathcal M(S\times V\times A)$ with the coarsest topology that makes the functional
\[
\mu \ \longmapsto\ \int g\,d\mu
\]
continuous for every bounded function $g:S\times V\times A\to\mathbb R$ such that: (i) for each $a\in A$, the function $(s,u)\mapsto g(s,u,a)$ is $\mathcal S\otimes\mathcal V$-measurable on $S\times V$;
(ii) for each $(s,u)\in S\times V$, the function $a\mapsto g(s,u,a)$ is continuous on $A$.

Consider the projection
\[
\mathcal M(S\times V\times A)\longrightarrow \mathcal M(S\times A),\qquad 
\mu\ \longmapsto\ \mu|_{S\times A}.
\]
For any bounded Carath\'eodory $h:S\times A\to\mathbb R$, set $g(s,u,a):=h(s,a)$. Then
\[
\int_{S\times V\times A} g\,d\mu=\int_{S\times A} h\,d(\mu|_{S\times A}),
\]
and hence $\mu\ \longmapsto\ \mu|_{S\times A}$ is continuous. Similarly the projection
\(
\mathcal M(S\times V\times A)\rightarrow \mathcal M(V\times A),\  
\mu\ \mapsto\ \mu|_{V\times A},
\)
is also continuous.

Let $\tau:=\lambda G^{-1}\in\mathcal M(S\times V)$. Consider
\[
P_\tau:=\{\mu\in\mathcal M(S\times V\times A):\ \mu|_{S\times V}=\tau\},
\]
endowed with the subspace topology inherited from $\mathcal M(S \times V \times A).$
By the disintegration theorem \cite[Corollary~A5]{Valadier1990-wx} (see also \cite[Theorem~3.4]{Kallenberg1997-dl}), for each $\mu \in P_\tau$ there exists a $\tau$-a.e. unique
transition probability $\kappa$ from $S \times V$ to $A$ such that
\[
  \mu(E \times B)=\int_E \kappa_{(s,u)}(B) d\tau (s,u), \ E \in \mathcal S \otimes \mathcal B(V), B \in \mathcal B(A).
\]
Conversely, any transition probability
$\kappa$ from $S\times V$ to $A$ induces an element $\mu\in P_\tau$
via this construction. Moreover, for every bounded function $g:S \times V \times A \to \mathbf{R}$ satisfying the above (i) and (ii),
\[
  \int_{S\times V\times A}\phi \, d\mu
  \;=\;
  \int_{S\times V}\Big(\int_A \phi(s,u,a)\,\kappa_{(s,u)}(da)\Big)\,\tau(d(s,u)).
\]
Therefore $P_\tau$ is topologically homeomorphic to the space $\mathcal R$ of transition probabilities from $S\times V$ to $A$
under the identification.
The space $\mathcal R$ is a compact subset of a Hausdorff locally convex topological vector space \cite[Theorem~2.3 (a)]{Balder1988-fb}. Hence, $\mathcal P_\tau$ can also be regarded as a compact subset of such a space. Moreover, it is clear that $\mathcal P_\tau$ is nonempty and convex.

For \(\rho\in\mathcal M(S\times A)\), define
\[
C_\rho:=\bigl\{(u,a)\in V\times A:\ u(a,\rho)\ge u(x,\rho)\ \ \forall x\in A\bigr\},
\qquad
B_\mu:=S\times C_{\mu|_{S\times A}} \in \mathcal S \otimes \mathcal B(V) \otimes \mathcal B(A) .
\]
Fix a countable dense subset \(D\subset A\) and set
\[
C_{\rho,x}:=\bigl\{(u,a): u(a,\rho)\ge u(x,\rho)\bigr\}\quad(x\in D).
\]
By continuity of each \(u \in V\), we have
\[
C_\rho=\bigcap_{x\in D}C_{\rho,x}.
\]
Define the correspondence \(\Gamma: P_\tau\rightrightarrows P_\tau\) by
\[
\Gamma(\mu):=\bigl\{\mu'\in P_\tau:\ \mu'(B_\mu)=1\bigr\}\qquad(\mu\in P_\tau).
\]
It is clear that $\Gamma$ is nonempty, convex-valued. 

We prove that $\Gamma$ has a closed graph.
Let \((\mu_\alpha,\mu_\alpha')\to(\mu,\mu')\) be a net with \(\mu_\alpha'\in\Gamma(\mu_\alpha)\).
By continuity of marginals,
\[
\mu_\alpha'|_{V\times A}\ \Rightarrow\ \mu'|_{V\times A},
\qquad
\mu_\alpha|_{S\times A}\ \to\ \mu|_{S\times A}.
\]
Fix \(x\in D\) and consider
\[
F_x:=\bigl\{\,(u,a,\rho)\in (V\times A)\times\mathcal M(S\times A):\ u(x,\rho)-u(a,\rho)\le 0\,\bigr\}.
\]
This is closed as the function $(u,a,\rho) \mapsto u(x,\rho)-u(a,\rho)$ is continuous. For each \(\rho\), it holds that \(C_{\rho,x} = \{(u,a) \in V \times A: (u,a,\rho) \in F_x\}\).
Thus, applying Lemma~\ref{lem:Port} (ii) to \(X=V\times A\), \(Z=\mathcal M(S\times A)\), \(F=F_x\), we obtain
\[
\limsup_\alpha\,\mu_\alpha'|_{V\times A}\bigl(C_{\mu_\alpha|_{S\times A},x}\bigr)\ \le\ \mu'|_{V\times A}\bigl(C_{\mu|_{S\times A},x}\bigr).
\]
Since \(\mu_\alpha'\in\Gamma(\mu_\alpha)\), the left-hand side equals \(1\); hence
\[
\mu'|_{V\times A}\bigl(C_{\mu|_{S\times A},x}\bigr)=1\qquad(\forall x\in D).
\]
Because \(D\) is countable,
\[
\mu'|_{V\times A}\bigl(C_{\mu|_{S\times A}}\bigr)
=\mu'|_{V\times A}\Bigl(\bigcap_{x\in D}C_{\mu|_{S\times A},x}\Bigr)
=1.
\]
    Hence \(\mu'(B_\mu)=1\), and \(\Gamma\) has a closed graph.

The set \( P_\tau\) is regarded as a nonempty, convex, compact subset in a Hausdorff locally convex topological vector space; \(\Gamma\) is nonempty, convex-valued, with closed graph.
Hence, by the Kakutani--Fan--Glicksberg fixed point theorem \cite[Corollary~17.55]{Aliprantis2006-hf} there exists \(\mu\in P_\tau\) with
\(
\mu\in\Gamma(\mu).
\)
Since \(\mu\in P_\tau\), its \(S\times V\)-marginal is \(\tau\), the distribution of the game \(G\).
By Theorem~\ref{the:saturation}, there exists a \(\mathcal T\)-measurable function \(f:T\to A\) such that $\mu = \lambda(G,f)^{-1}$. In particular, the \(S\times A\)-marginal of \(\mu\) is the distribution of \((\alpha,f)\).
Since \(\mu(B_\mu)=1\) and \( \mu =\lambda  (G, f)^{-1}\), the function \(f\) satisfies \(\lambda\text{-a.e.\ }t,\)
\[
v(t)\bigl(f(t),\,\mu|_{S\times A}\bigr)\ \ge\ v(t)\bigl(x,\,\mu|_{S\times A}\bigr)
\qquad
(\ \forall x\in A),
\]
hence \(f\) is a pure-strategy Nash equilibrium of the game \(G\).

\end{proof}

Take the trait space to be the agent space: $(S,\mathcal S)=(T,\mathcal F)$ and $\alpha$ the identity $\mathrm{id}_T$. Then, we obtain the existence of a pure strategy equilibrium in semi-anonymous settings:
  \begin{corollary}\label{coro:semi}
      Assume that the $\sigma$-algebra $\mathcal T$ is nowhere equivalent to the sub-$\sigma$-algebra $\mathcal F$. Then, for any $\mathcal F$-measurable game function $v:T\to \mathcal V_{(A,T)}$, there exists an
    $\mathcal T$-measurable $f:T\to A$ such that for $\lambda$-a.e.\ $t\in T$,
    \[
      v_t\!\big(f(t),\,\lambda(\mathrm{id}_T,f)^{-1}\big)\ \ge\
      v_t\!\big(a,\,\lambda(\mathrm{id}_T,f)^{-1}\big)\quad \text{for all } a\in A .
    \]
  \end{corollary}

\begin{remark}
The nonexistence problem is first resolved in \cite{Khan1999-gy} by working with a hyperfinite Loeb agent space.
It has been shown in \cite{Keisler2009-ib} that the saturation property is necessary and sufficient for the existence of pure-strategy Nash equilibria in large games (without traits) when the action space is any fixed uncountable compact metric space.
Assuming the trait space is Polish, this saturation characterization extends to large games with traits~\cite{Khan2013-pk}. 
If the trait space is compact metric and $\mathcal{F}$ is countably generated, equilibrium existence is also characterized by nowhere equivalence for large games with traits and a fixed uncountable action space; see \cite[Theorem~2]{He2017-cr} and \cite[Theorem~2]{He2018-gj}. If the action space is finite, the game function may be $\mathcal T$-measurable in this nowhere-equivalence characterization; see \cite{Wu2022-as}.
\end{remark}

\begin{remark}
The converse of Theorem~\ref{the:Nash} does not hold in general. Indeed, when the trait space $S$ is a singleton, large games with traits reduce to the usual large games (without traits). Hence, if $(T,\mathcal{T},\lambda)$ is a saturated probability space and $\mathcal{F}=\mathcal{T}$, then every $\mathcal{F}$-measurable game with traits admits a $\mathcal{T}$-measurable Nash equilibrium \cite[Theorem~4.6]{Keisler2009-ib}. Nevertheless, $\mathcal{T}$ is not nowhere equivalent to $\mathcal{F}$.
By contrast, it remains open whether the converse of Corollary \ref{coro:semi} holds.
\end{remark}

\section{Proof of Theorem~\ref{the:main}}
\label{sec:Proof}
\subsection{Proof of the sufficiency part of Theorem~\ref{the:main}}
First, we present the following lemma. Its proof is standard, but we include it for completeness, as we will use the lemma repeatedly below.
\begin{lemma}\label{lem:integration}
    Let $X$ be a Polish space and $(T,\mathcal T,\lambda)$ be a probability space and $f:(T,\mathcal T)\to(X,\mathcal B(X))$ be measurable. Let $\mathcal F\subset\mathcal T$ be a sub-$\sigma$-algebra and $\mu^{\,f\mid\mathcal F}$ be a regular conditional distribution of $f$ given $\mathcal F$.
If $\phi:T\times X\to\mathbb R$ is $\mathcal F\otimes\mathcal B(X)$-measurable and either nonnegative or bounded, then
\[
\int_T\!\left(\int_X \phi(t,x)\,\mu^{\,f\mid\mathcal F}(t, dx)\right)d\lambda(t)
= \int_T \phi\bigl(t,f(t)\bigr)d\lambda(t).
\]
\end{lemma}

\begin{proof}
It suffices to prove the case $\phi=\mathbf 1_E$ with $E\in\mathcal F\otimes\mathcal B(X)$. Indeed, if $\phi\ge 0$, approximate $\phi$ by an increasing sequence of $\mathcal F\otimes\mathcal B(X)$-measurable simple functions $\phi_n\uparrow\phi$ and apply the monotone convergence theorem; if $\phi$ is bounded and real-valued, write $\phi=\phi^+-\phi^-$ and apply the nonnegative case to $\phi^\pm$, then take the difference.

Thus we prove the claim for indicator functions.
Define
\[
\mathcal C
:= \Bigl\{E\in \mathcal F\otimes\mathcal B(X):
\int_T\!\int_X \mathbf 1_E(t,x)\,\mu^{\,f\mid\mathcal F}(t, dx)\,d\lambda(t)
= \int_T \mathbf 1_E\bigl(t,f(t)\bigr)d\lambda(t)\Bigr\}.
\]
Then, $\mathcal C$ clearly satisfies:
(1) $T\times X\in\mathcal C$;
(2) if $A,B\in\mathcal C$ with $A\subset B$, then $B\setminus A\in\mathcal C$;
(3) if $\{E_n\}\subset\mathcal C$ are pairwise disjoint, then $\bigcup_n E_n\in\mathcal C$.
Hence $\mathcal C$ is a Dynkin family.

Consider the class of rectangles
\[
\mathcal P:=\{A\times B:\ A\in\mathcal F,\ B\in\mathcal B(X)\}.
\]
Then, $\mathcal P$ is a $\pi$-system with $\sigma(\mathcal P)=\mathcal F\otimes\mathcal B(X)$. For each rectangle $E=A\times B$, by the definition of the regular conditional distribution,
\[
\int_T\!\int_X \mathbf 1_{A\times B}(t,x)\,\mu^{f\mid\mathcal F}(t, dx)\,\mathrm d\lambda(t)
= \int_A \,\mu^{\,f\mid\mathcal F}_t(B)d\lambda(t)
= \lambda (\{f \in B\} \cap A  )
= \int_T \mathbf 1_{A\times B}\bigl(t,f(t)\bigr)d\lambda(t),
\]
so $A\times B\in\mathcal C$, i.e., $\mathcal P\subset\mathcal C$. By the $\pi$–$\lambda$ theorem \cite[Theorem~3.2]{Billingsley1986-wb}, $\sigma(\mathcal P)\subset\mathcal C$, hence $\mathcal F\otimes\mathcal B(X)\subset\mathcal C$.
This completes the proof.
\end{proof}

We first prove that NE $\Rightarrow$ B5.
\begin{proof}[Proof of NE $\Rightarrow$ B5]
    We assume nowhere equivalence of $\mathcal T$ to $\mathcal F$. Since $\mathcal T$ is also nowhere equivalent to $\mathcal G$,
there exists an atomless independent supplement
  $\mathcal H$ of $\mathcal G.$
Because $(T,\mathcal H,\lambda)$ is atomless, there exists
an $\mathcal H$-measurable $U:T\to[0,1]$ such that
\[
\lambda U^{-1}=m,
\]
where $m$ is the Lebesgue measure on $[0,1]$; see \cite[Lemma~2.1 (ii)]{Keisler2009-ib}.

Since $X$ is Polish and $G$ is a $\mathcal G$-measurable transition probability to $X$, by Lemma 4.22 in \cite{Kallenberg1997-dl} there exists a
$(\mathcal G\otimes\mathcal B([0,1]))$-measurable map
\[
f:T\times[0,1]\to X
\]
such that for every $t\in T$, the pushforward of $f(t,U(\cdot))$ satisfies:
\begin{equation} \label{eq:kernel}
\lambda f(t,U(\cdot))^{-1}=G(t,\cdot).
\end{equation}
Define
\[
g(t):=f\bigl(t,U(t)\bigr)\qquad(t\in T).
\]
Because $f$ is $\mathcal G\otimes\mathcal B([0,1])$-measurable and $U$ is $\mathcal H$-measurable,
$g$ is measurable with respect to $\sigma(\mathcal G\cup\mathcal H)\subset\mathcal T$; hence $g$ is
$\mathcal T$-measurable.

Consider the product measure $\lambda\otimes\lambda$ on $\mathcal G\otimes\mathcal H$.
By the independence of $\mathcal G$ and $\mathcal H$, the pushforward of $\lambda$ under the diagonal map
$\Delta:T\to T\times T$, $\Delta(t)=(t,t)$, satisfies
\[
\lambda\Delta^{-1}=\lambda\otimes\lambda\qquad\text{on }\mathcal G\otimes\mathcal H.
\]
Fix $E\in\mathcal G$ and $B\in\mathcal B(X)$, and set
\[
S:=\{(t_1,t_2): f(t_1,U(t_2))\in B\}.
\]
Then
\[
\lambda\bigl(E\cap g^{-1}(B)\bigr)
= \bigl(\lambda\Delta^{-1}\bigr)\bigl((E\times T)\cap S\bigr)
= (\lambda\otimes\lambda)\bigl((E\times T)\cap S\bigr).
\]
By Fubini's theorem,
\begin{align*}
(\lambda\otimes\lambda)\bigl((E\times T)\cap S\bigr)
&= \int_E \lambda\!\left(\{t_2: f(t,U(t_2))\in B\}\right)\,d\lambda(t)\\
&= \int_E \lambda  f\bigl(t,U(\cdot)\bigr) ^{-1}(B)\,d\lambda(t)\\
&= \int_E G(t,B)\, d\lambda(t),
\end{align*}
where the last equality uses \eqref{eq:kernel}. Hence, we have 
\[\lambda\bigl(E\cap g^{-1}(B)\bigr) = \int_E G(t,B)\, d\lambda(t).
\]
By the definition of regular conditional distributions,
this identity for all $E\in\mathcal G$ and $B\in\mathcal B(X)$ implies $\mu^{g\mid\mathcal G}=G$.

\end{proof}

\begin{proof}[Proof of NE $\Rightarrow$ B1]

Since we have already proved NE $\Rightarrow$ B5, it now suffices to prove B5 $\Rightarrow$ B1.
Take two $\mathcal T$-measurable selections $f_1,f_2$ of $F$ and $\alpha\in(0,1).$ We prove that $\alpha\,\mu^{f_1\mid\mathcal G}
 +(1-\alpha)\,\mu^{f_2\mid\mathcal G} \in \mathcal{R}_F^{(\mathcal T, \mathcal G)}.$
 
First, define the $\mathcal F$-measurable transition probability
\[
H_{\mathcal F}:=\alpha\,\mu^{f_1\mid\mathcal F}
+(1-\alpha)\,\mu^{f_2\mid\mathcal F}.
\]
By B5 there exists a $\mathcal T$-measurable $f:T\to X$ with
\[
\mu^{f\mid\mathcal F}=H_{\mathcal F}. \tag{$\ast$}
\]

Consider the $\mathcal F\otimes\mathcal B(X)$-measurable normal integrand\footnote{For any probability space $(T, \mathcal T, \lambda)$ and any Polish space $X$, a function $f:T\times X \to \mathbf{R}\cup\{+\infty\} $ is \emph{normal integrand} if (i) $f$ is $\mathcal T \otimes \mathcal B(X)$-measurable; (ii) for each $t \in T,$ the function $x \mapsto f(t,x)$ is lower semicontinuous.}
\[
I_F(t,x):=\begin{cases}
0,& (t,x)\in\mathrm{Gr}(F),\\[2pt]
+\infty,& (t,x)\notin\mathrm{Gr}(F).
\end{cases}
\]
Since $f_i(t)\in F(t)$ a.e., we have
\[
\int_T I_F\!\bigl(t,f_i(t)\bigr)\,d\lambda(t)=0\qquad (i=1,2).
\]
By Lemma~\ref{lem:integration}, we compute
\[
\begin{aligned}
\int_T I_F\!\bigl(t,f(t)\bigr)\,d\lambda(t)
&= \int_T \!\int_X I_F(t,x)\,\mu^{f\mid\mathcal F}(t,dx) d\lambda(t) \\
&= \int_T \!\int_X I_F(t,x)\,H_{\mathcal F}(t,dx) d\lambda(t) \\
&= \alpha \int_T \!\int_X I_F(t,x)\,\mu^{f_1\mid\mathcal F}(t,dx) d\lambda(t)
+(1-\alpha)\int_T \!\int_X I_F(t,x)\,\mu^{f_2\mid\mathcal F}(t,dx) d\lambda(t)\\
&= \alpha \int_T I_F\!\bigl(t,f_1(t)\bigr)\,d\lambda(t)
  +(1-\alpha)\int_T I_F\!\bigl(t,f_2(t)\bigr)\,d\lambda(t)\\
&= 0.
\end{aligned}
\]
Since $I_F\ge 0$, it follows that
\[
 I_F\!\bigl(t,f(t)\bigr)=0 \text{ a.s.}
\]
Thus $f(t)\in F(t)$ a.e.; i.e., $f$ is a $\mathcal T$-measurable selection of $F$.

For every $G\in\mathcal G$ and $B\in\mathcal B(X)$, since $G \in \mathcal F$, we have
\begin{align*}
\lambda(\{f\in B\}\cap G)
&= \int_G \mu^{f\mid\mathcal F}(t,B)\,d\lambda(t) \\[2mm]
&= \int_G \Bigl(\alpha\,\mu^{f_1\mid\mathcal F}(t,B)
               +(1-\alpha)\,\mu^{f_2\mid\mathcal F}(t,B)\Bigr)\,d\lambda(t) \\[1mm]
&= \alpha \int_G \mu^{f_1\mid\mathcal F}(t,B)\,d\lambda(t)
  +(1-\alpha)\int_G \mu^{f_2\mid\mathcal F}(t,B)\,d\lambda(t) \\[1mm]
&= \alpha\,\lambda(\{f_1\in B\}\cap G)
  +(1-\alpha)\,\lambda(\{f_2\in B\}\cap G) \\[1mm]
&= \alpha \int_G \mu^{f_1\mid\mathcal G}(t,B)\,d\lambda(t)
  +(1-\alpha)\int_G \mu^{f_2\mid\mathcal G}(t,B)\,d\lambda(t) \\[1mm]
&= \int_G \Bigl(\alpha\,\mu^{f_1\mid\mathcal G}(t,B)
               +(1-\alpha)\,\mu^{f_2\mid\mathcal G}(t,B)\Bigr)\,d\lambda(t).
\end{align*}
Since this holds for every $G\in\mathcal G$ and $B\in\mathcal B(X)$, by the defining property of
regular conditional distributions we conclude
\[
\alpha\,\mu^{f_1\mid\mathcal G}
 +(1-\alpha)\,\mu^{f_2\mid\mathcal G} = \mu^{f\mid\mathcal G} \in \mathcal{R}_F^{(\mathcal T, \mathcal G)}.
\]

\end{proof}

\begin{proof}[Proof of NE $\Rightarrow$ B2]
Again, it suffices to prove B5 $\Rightarrow$ B2.
Consider a net $(\nu_\alpha)_\alpha$ in \(
R^{(T,\mathcal G)}_F
\)
and assume $\nu_\alpha\to \nu \in \mathcal R^\mathcal G$ in the weak topology. By definition of $R^{(T,\mathcal G)}_F$, for each $\alpha$ there exists a $\mathcal T$-measurable selection $f_\alpha:T\to X$ such that
\[
\nu_\alpha=\mu^{f_\alpha\mid\mathcal G}.
\]
For the same $f_\alpha$, let
\[
\rho_\alpha:=\mu^{f_\alpha\mid\mathcal F}.
\]

Embed $X$ as a Borel subset of a compact metric space $\widehat X$ (see \cite[Theorem~1.8]{Kallenberg1997-dl}) and henceforth identify $X\subset \widehat X$. Whenever we regard $\rho_\alpha$ as a transition probability to $\widehat X$, we write $\widehat\rho_\alpha$:
\[
\widehat\rho_\alpha(t,A):=\rho_\alpha\bigl(t,A\cap X\bigr),\qquad t \in T, A\in\mathcal B(\widehat X).
\]
Then $(\widehat\rho_\alpha)_\alpha$ is a net in $R^{\mathcal F}(\widehat X)$. Since $\widehat X$ is compact metric, the space $R^{\mathcal F}(\widehat X)$ endowed with the weak topology is compact \cite[Theorem~2.3 (a)]{Balder1988-fb}. Therefore $(\widehat\rho_\alpha)_\alpha$ admits a weakly convergent subnet:
\[
\widehat\rho_\beta \ \Longrightarrow\ \widehat\rho\quad\text{in }R^{\mathcal F}(\widehat X).
\]

We next prove that for $\lambda$-a.e.\ $t\in T$, the probability measure $\widehat\rho(t,\cdot)$ concentrates on $X$; that is,
\[
\widehat\rho\bigl(t,X)=1\quad(\lambda\text{-a.e.\ }t).
\]
For a transition probability $\kappa$ from $T$ to $\widehat X$ and $A\in\mathcal B(X)$, define the \emph{marginal} $\kappa^X$ \emph{on} $X$ by
\[
\kappa^{X}(A):=\int_T \kappa(t,A)\,d\lambda(t).
\]
We note that for every bounded continuous function $\varphi:X \to \mathbf{R}$, it holds
\begin{equation}\label{eq:marginal}
    \int_X \varphi\, d\kappa^{X}
=\int_T\!\int_X \varphi(x)\,\kappa(t,dx)\,d\lambda(t).
\end{equation}

For each $A\in\mathcal B(X)$,
\[
(\widehat\rho_\beta)^{X}(A)
=\int_T \widehat\rho_\beta(t,A)\,d\lambda(t)
=\int_T \rho_\beta(t,A)\,d\lambda(t)
=\int_T \mu^{f_\beta \mid\mathcal F}(t,A) \,d\lambda
=\lambda (\{f_\beta\in A\}),
\]
hence $(\widehat\rho_\beta)^{X}=\lambda  f_\beta^{-1}$.
Similarly,
\[
\nu_\beta^{X}(A)=\int_T \nu_\beta(t,A)\,d\lambda(t)= \int_T \mu^{f_\beta \mid\mathcal G}(t,A) \,d\lambda = 
\lambda (\{f_\beta\in A\})\,
\]
so $\nu_\beta^{X}=\lambda  f_\beta^{-1}$.

From $\widehat\rho_\beta\Longrightarrow \widehat\rho$ and the identity \eqref{eq:marginal}, for any bounded continuous function $\varphi:X \to \mathbf{R}$,
\[
\int_X \varphi\, d(\widehat\rho_\beta)^{X}
=\int_T\!\int_X \varphi\, \widehat\rho_\beta(t,dx)\,d\lambda
\;\longrightarrow\;
\int_T\!\int_X \varphi\, \widehat\rho(t,dx)\,d\lambda
=\int_X \varphi\, d(\widehat\rho)^{X}.
\]
Thus $(\widehat\rho_\beta)^{X}\Rightarrow (\widehat\rho)^{X}$ in $\mathcal M(X)$. Likewise, $\nu_\beta \Rightarrow \nu$ implies $\nu_\beta^{X}\Rightarrow \nu^{X}$.
By the uniqueness of weak limits, $(\widehat\rho)^{X}=\nu^{X}$. In particular,
\[
\int_T \widehat\rho(t,X)\,d\lambda(t)
=(\widehat\rho)^{X}(X)= \nu^X(X) = 1,
\]
whence $\widehat\rho\bigl(t, X\bigr)=1$ for $\lambda$-a.e.\ $t$.

Next, regard $\widehat\rho$ as a transition probability into $X$ and denote it by $\rho$; that is,
\[
\rho(t,A):=\widehat\rho(t,A),\qquad A\in\mathcal B(X).
\]
Here, \(\widehat\rho_\beta \Rightarrow \widehat\rho\) means that, when we view
\(\rho_\beta\) and \(\rho\) as transition probabilities to \(\widehat X\), the net
\((\rho_\beta)_\beta\) converges weakly to \(\rho\) in \(\mathcal R^{\mathcal F}(\widehat X)\).
Hence, it follows from Theorem~2.1 in \cite{Balder1988-fb} that
\[
\rho_\beta \ \Longrightarrow\ \rho \quad\text{in } \mathcal R^{\mathcal F}(X).
\]
By B5 there exists a $\mathcal T$-measurable map $f:T\to X$ with
\[
\rho=\mu^{f\mid\mathcal F}.
\]
Therefore $\mu^{f_\beta\mid\mathcal F}\Rightarrow \mu^{f\mid\mathcal F}$ in $\mathcal R^{\mathcal F}(X)$.

Consider the $\mathcal F\otimes\mathcal B(X)$-measurable normal integrand
\[
I_F(t,x):=\begin{cases}
0,& (t,x)\in \operatorname{Gr}(F),\\[2mm]
+\infty,& (t,x)\notin \operatorname{Gr}(F).
\end{cases}
\]
For each $\beta$, Since $f_\beta(t)\in F(t)$ a.e., we have
\[
\int_T I_F\!\bigl(t,f_\beta(t)\bigr)\,d\lambda(t)=0.
\]
The functional
\[
\Phi(\nu):=\int_T \int_X I_F(t,x)\,\nu(t,dx)\,d\lambda(t),
\qquad \nu\in R^{\mathcal F},
\]
is lower semicontinuous for the weak topology of $\mathcal R^\mathcal F$ \cite[Theorem~2.2 (a)]{Balder1988-fb}. Hence, by Lemma~\ref{lem:integration},
\[
0 \le \Phi(\mu^{f\mid\mathcal F})
\;\le\; \liminf_{\beta}\,\Phi(\mu^{f_\beta\mid\mathcal F}) \;= \; \liminf_{\beta}\, \int_T I_F\!\bigl(t,f_\beta(t)\bigr)\,d\lambda(t)\; = \; 0,
\]
so $\Phi(\mu^{g\mid\mathcal F})=0$. Again, by Lemma~\ref{lem:integration}, this forces
\[
I_F(t,f(t))=0 \quad \text{a.s.}, \quad \text{i.e., } f(t)\in F(t)\ \text{a.s.}
\]
Thus $f$ is a $\mathcal T$-measurable selection of $F$.

Let $\phi : T \times X \to \mathbb{R}$ be a bounded $\mathcal G$-Carath\'eodory function. Since $\phi$ is also a $\mathcal F$-Carath\'eodory function, for each $\beta$,
\begin{equation}\label{eq:caratheodory_B2}
\begin{aligned}
\int_T\!\int_X \phi(t,x)\,\mu^{f_\beta\mid\mathcal G}(t,dx)\,d\lambda(t)
&= \int_T \phi(t,f_\beta(t))\,d\lambda \\
&= \int_T\!\int_X \phi(t,x)\,\mu^{f_\beta\mid\mathcal F}(t,dx)\,d\lambda(t).
\end{aligned}
\end{equation}
Since $\mu^{f_\beta\mid\mathcal F}\Rightarrow \mu^{f\mid\mathcal F}$, the right-hand side of \eqref{eq:caratheodory_B2} converges to
\[
\int_T\!\int_X \phi(t,x)\,\mu^{f\mid\mathcal F}(t,dx)\,d\lambda(t).
\]
This limit can be written as:
\begin{align*}
    \int_T\!\int_X \phi\,d\mu^{f\mid\mathcal F}\,d\lambda
&= \int_T \phi(t,f(t))\,d\lambda\\
&= \int_T\!\int_X \phi\,d\mu^{f\mid\mathcal G}\,d\lambda.
\end{align*}
Therefore,
\[
\int_T\!\int_X \phi\,d\mu^{f_\beta\mid\mathcal G}\,d\lambda
\;\longrightarrow\;
\int_T\!\int_X \phi\,d\mu^{f\mid\mathcal G}\,d\lambda,
\]
for every bounded $\mathcal G$-Carath\'eodory $\phi$, so $\mu^{f_\beta\mid\mathcal G} \Rightarrow \mu^{f|\mathcal G}$ in $\mathcal R^\mathcal G(X).$ 
Hence, $\nu = \mu^{f|\mathcal G} \in \mathcal R^{(T,\mathcal G)}_F.$
\end{proof}

\begin{proof}[Proof of NE $\Rightarrow$ B3]

We prove that B5 \(\Rightarrow\) B3.
Take a net \((\mu_\alpha)_\alpha\) in \(\mathcal R_F^{(T,\mathcal G)}\).
To obtain B3, it is enough to show that \((\mu_\alpha)_\alpha\) admits a convergent subnet whose limit still belongs to \(\mathcal R_F^{(T,\mathcal G)}\).
By the definition of \(\mathcal R_F^{(T,\mathcal G)}\), for each \(\alpha\) there exists a \(\mathcal T\)-measurable selection
\(f_\alpha\) of $F$ such that
\[
\mu_\alpha=\mu^{f_\alpha\mid\mathcal G}.
\]

Since \(X\) is a Polish space, there exists a Borel isomorphic embedding of \(X\) into a compact metric space \(\widehat X\) (see \cite[Theorem~1.8]{Kallenberg1997-dl}) and, from now on, we identify \(X\) with a Borel subset of \(\widehat X\).
Define
\[
  \widehat F:T\rightrightarrows \widehat X,\qquad \widehat F(t):=F(t)\subset\widehat X .
\]
Since \(F\) is compact-valued in \(X\), each \(\widehat F(t)\) is closed in \(\widehat X\).
Moreover, the graph is measurable: \(\operatorname{Gr}(\widehat F)\in \mathcal F\otimes \mathcal B(\widehat X)\).

For each index \(\alpha\) and each \(B\in\mathcal B(\widehat X)\), define 
\[
  \widehat\mu_\alpha(t,B)\;:=\;\mu_\alpha\bigl(t, B\cap X\bigr).
\]
Then \((\widehat\mu_\alpha)_\alpha\) is a net in \(\mathcal R^{\mathcal G}(\widehat X)\).
Since \(\mathcal R^{\mathcal G}(\widehat X)\) is compact in the weak topology \cite[Theorem~2.3 (a)]{Balder1988-fb}, there exists a subnet
\(\widehat\mu_\beta \Rightarrow \widehat\mu\) for some \(\widehat\mu\in\mathcal R^{\mathcal G}(\widehat X)\).

We show that for \(\lambda\)-a.e.\ \(t\), the probability measure \(\widehat\mu(t, \cdot)\) concentrates on \(\widehat F(t)\).
Because \(\widehat F\) is closed-valued and measurable, the function
\[
  I_{\widehat F}(t,x):=\begin{cases}
    0, & (t,x)\in \operatorname{Gr}(\widehat F),\\[2pt]
    +\infty, & \text{otherwise}
  \end{cases}
\]
is a \(\mathcal F \otimes \mathcal B (\widehat X)\)-measurable normal integrand.
Define the functional on transition probabilities to \(\widehat X\) by
\[
  \Phi(\nu)\;:=\;\int_T\!\int_{\widehat X} I_{\widehat F}(t,x)\,\nu(t,dx)\,d\lambda(t),
  \qquad \nu\in \mathcal R^{\mathcal F}(\widehat X).
\]
Then \(\Phi\) is weakly lower semicontinuous on \(\mathcal R^{\mathcal G}(\widehat X)\); see \cite[Theorem~2.2 (a)]{Balder1988-fb}.

For each \(\beta\), since \(f_\beta(t)\in F(t)\) a.e.\ and \(\widehat\mu_\beta(t, X)=1\) for all \(t\), by Lemma~\ref{lem:integration} we have
\[
\begin{aligned}
\Phi(\widehat\mu_\beta)
&= \int_T\!\int_{\widehat X} I_{\widehat F}(t,x)\,\widehat\mu_\beta(t,dx)\,d\lambda(t)
 = \int_T\!\int_{X} I_{\widehat F}(t,x)\,\mu^{f_\beta\mid\mathcal G}(t,dx)\,d\lambda(t) \\
&= \int_T I_{\widehat F}\bigl(t,f_\beta(t)\bigr)\,d\lambda(t)
 = 0.
\end{aligned}
\]
By lower semicontinuity,
\[
  0\;\le\; \Phi(\widehat\mu)\;\le\;\liminf_\beta \Phi(\widehat\mu_\beta)\;=\;0,
\]
so \(\Phi(\widehat\mu)=0\).
Therefore 
\[
\int_{\widehat X} I_{\widehat F}(t,x)\,\widehat\mu(t,dx) = 0 
\quad\text{for }\lambda\text{-a.e.\ }t\in T.
\]
Since \(I_{\widehat F}(t,\cdot)\ge 0\) and \(I_{\widehat F}(t, \cdot)>0\) precisely on \(\widehat X\setminus \widehat F(t)\), it follows that
\[
  \widehat\mu\bigl( t, \widehat X\setminus \widehat F(t)\bigr)=0
  \quad\text{for }\lambda\text{-a.e.\ }t\in T.
\]
Thus the probability measure \(\widehat\mu(t, \cdot)\) concentrates on \(\widehat F(t)\) for \(\lambda\)-a.e.\ \(t\).

Define \(\mu\in\mathcal R^{\mathcal F}(X)\) as
\[
  \mu(t,A):=\widehat\mu(t,A)\qquad\bigl( t \in T, A\in\mathcal B(X)\bigr).
\]
Since \(\widehat\mu_\beta\Rightarrow \widehat\mu\) in \(\mathcal R^{\mathcal F}(\widehat X)\), Theorem~2.1 in \cite{Balder1988-fb} yields
\[
  \mu_\beta \Rightarrow\ \mu \ \text{in } \mathcal R^{\mathcal F}(X).
\]
By B5, there exists a \(\mathcal T\)-measurable function \(f:T\to X\) such that
\(
  \mu=\mu^{f\mid\mathcal F}.
\)
Therefore, 
\[
\mu^{f_\beta\mid\mathcal F}\Rightarrow \mu^{f\mid\mathcal F} \text{ in } R^{\mathcal F}(X).
\]
The same argument as in the proof of NE $\Rightarrow$ B2 yields $f$ is a $\mathcal T$-measurable selection of $F$ and
\[
\mu^{f_\beta\mid\mathcal G}\ \Rightarrow\ \mu^{f\mid\mathcal G}\  \in R_F^{(\mathcal T, \mathcal G)}.
\]
\end{proof}

\begin{proof}[Proof of NE $\Rightarrow$ B4] We prove B3 $\Rightarrow$ B4. Assume (B3).
Then, $H(z)=R^{(T,\mathcal G)}_{G_z}$ is weakly compact for every $z\in Z$; in particular, each $H(z)$ is closed in the weak topology. Therefore, to deduce that $H$ has a closed graph, it suffices to prove that $H$ is upper semicontinuous; see \cite[Theorem~17.11]{Aliprantis2006-hf}.

Fix $z_0\in Z$ and suppose, toward a contradiction, that $H$ is not upper semicontinuous at $z_0$. Then there exist a weakly open set $U$ of $\mathcal R^\mathcal G$ with $H(z_0)\subset U$ and a sequence $(z_n)_{n\ge1}$ with $z_n\to z_0$ together with
\[
\mu_n\in H(z_n)\setminus U\qquad(n=1,2,\dots).
\]
By the definition of $H(z_n)$, for each $n$ there exists an $\mathcal F$-measurable selection $f_n:T\to X$ of $G_{z_n}$ such that
\[
\mu_n=\mu^{f_n\,|\,\mathcal G}.
\]

Consider the sequence $(\mu^{f_n\mid\mathcal F})$ of RCDs given $\mathcal F$. Since $f_n$ is a measurable selection of $G_{z_n}$ and $G_{z_n}(t)\subset F(t)$, we have
\[
\mu^{f_n\mid\mathcal F}\ \in\ \mathcal R^{(T,\mathcal F)}_{F} \ (\forall n \in \mathbf{N}).
\]
By (B3), the set $\mathcal R^{(T,\mathcal F)}_{F}$ is compact in the weak topology; hence the sequence $\{\mu^{f_n\mid\mathcal F}\}$ admits at least one cluster point $\nu \in \mathcal R^{(T,\mathcal F)}_{F}$.
By definition of $\mathcal R^{(T,\mathcal F)}_{F}$, there exists an $\mathcal F$–measurable selection $f$ of $F$ such that
\(
\nu=\mu^{f\mid\mathcal F}.
\)

Define
\[
\Gamma_n(t)\ :=\ \overline{\bigcup_{m\ge n} G(t,z_m)}\qquad(n=1,2,\dots).
\]
For each $m$, $G(\cdot,z_m)$ is measurable; therefore, the countable union $\bigcup_{m\ge n} G(\cdot,z_m)$ is measurable as well. Consequently, the graph of $\Gamma_n$ is measurable with respect to $\mathcal F_\mu \otimes \mathcal B (X)$, where $\mathcal F_\mu \subset \mathcal T$ denotes the $\mu$-completion of $\mathcal F$; see \cite[Theorem~18.6 and 18.25]{Aliprantis2006-hf}.

Next, consider $\mu^{f_n\mid\mathcal F_\mu}$ and $\mu^{f\mid\mathcal F_\mu}$, the RCDs of $f_n$ and $f$ given $\mathcal F_\mu$. We note that $\mu^{f\mid\mathcal F_\mu}$ is again a cluster point of $(\mu^{f_n\mid\mathcal F_\mu}).$\footnote{Fix $n\in\mathbb N$ and a weak neighborhood $U$ of $\mu^{f\mid\mathcal F_\mu}$.
Choose sets $D_j\in\mathcal F_\mu$ and bounded continuous functions $c_j:X \to \mathbf{R}$, $j=1,\dots,k$, and pick
$\varepsilon>0$ so that the neighborhood
\[
V:=\Big\{\nu:\ \big|\textstyle\int_{D_j}\!\!\int_X c_j(x)\,\nu(t,dx)\,d\lambda
-\int_{D_j}\!\!\int_X c_j(x)\,\mu^{f\mid\mathcal F_\mu}(t,dx)\,d\lambda\big|<\varepsilon,\ \forall j\le k\Big\}
\]
satisfies $V\subset U$; see \cite[Theorem~2.2 (c)]{Balder1988-fb}.
By the definition of the completion, write $D_j=D'_j\cup N_j$ with
$D'_j\in\mathcal F$ and $\lambda(N_j)=0$. Using the same $\{c_j\}$, define a
neighborhood of $\mu^{f\mid\mathcal F}$ by
\[
V^{\mathcal F}:=\Big\{\nu:\ \big|\textstyle\int_{D'_j}\!\!\int_X c_j(x)\,\nu(t,dx)\,d\lambda
-\int_{D'_j}\!\!\int_X c_j(x)\,\mu^{f\mid\mathcal F}(t,dx)\,d\lambda\big|<\varepsilon,\ \forall j\le k\Big\}.
\]
Since $\mu^{f\mid\mathcal F}$ is a cluster point of $(\mu^{f_n\mid\mathcal F})$,
there exists $m\ge n$ with $\mu^{f_m\mid\mathcal F}\in V^{\mathcal F}$.
For each $j$, apply Lemma~\ref{lem:integration} to obtain
\[
\int_{D'_j}\!\!\int_X c_j\,d\mu^{f\mid\mathcal F}\,d\lambda
=\int_{D'_j}\! c_j\!\big(f(t)\big)\,d\lambda
=\int_{D_j}\! c_j\!\big(f(t)\big)\,d\lambda
=\int_{D_j}\!\!\int_X c_j\,d\mu^{f\mid\mathcal F_\mu}\,d\lambda.
\]
The same identities hold with $f_m$
replacing $f$.
Hence, for all $j\le k$,
\[
\Big|\textstyle\int_{D_j}\!\!\int_X c_j\,d\mu^{f_m\mid\mathcal F_\mu}\,d\lambda
-\int_{D_j}\!\!\int_X c_j\,d\mu^{f\mid\mathcal F_\mu}\,d\lambda\Big|
=
\Big|\textstyle\int_{D'_j}\!\!\int_X c_j\,d\mu^{f_m\mid\mathcal F}\,d\lambda
-\int_{D'_j}\!\!\int_X c_j\,d\mu^{f\mid\mathcal F}\,d\lambda\Big|
<\varepsilon,
\]
so $\mu^{f_m\mid\mathcal F_\mu}\in V\subset U$.
Since $U$ and $n$ were arbitrary, $\mu^{f\mid\mathcal F_\mu}$ is a cluster point of
$(\mu^{f_n\mid\mathcal F_\mu})$.}
For arbitrary $n$, define the $(\mathcal F_\mu\otimes\mathcal B(X))$–measurable normal integrand
\[
I_n(t,x):=\begin{cases}
0,& x\in\Gamma_n(t),\\[2pt]
+\infty,& x\notin\Gamma_n(t).
\end{cases}
\]
For an $\mathcal F_\mu$–measurable transition probability $\mu \in \mathcal R^{\mathcal F_\mu}$, set
\[
\Phi_n(\mu):=\int_T\!\int_X I_{n}(t,x)\,\mu(t,dx)\,\lambda(dt).
\]
By the standard lower semicontinuity theorem for integral functionals generated by normal integrands \cite[Theorem~2.2 (a)]{Balder1988-fb}, for every $\varepsilon>0$ there exists a weak neighborhood $\mathcal U_{n,\varepsilon}$ of $\mu^{f\mid\mathcal F_\mu}$ such that
\[
\Phi_n(\nu) > \Phi_n(\mu^{f\mid\mathcal F_\mu})-\varepsilon
\qquad\text{for all }\nu\in\mathcal U_{n,\varepsilon}.
\]
Since $\mu^{f\mid\mathcal F_\mu}$ is a cluster point of $(\mu^{f_n\mid\mathcal F_\mu})_{n\ge1}$, there exists $m\ge n$ with $\mu^{f_m\mid\mathcal F_\mu}\in\mathcal U_{n,\varepsilon}$. Because $f_m$ is a selection of $\Gamma_n$, Lemma~\ref{lem:integration} yields
\[
\Phi_n(\mu^{f_m\mid\mathcal F_\mu})
=\int_T\!\Big(\int_X I_n(t,x)\,\mu^{f_m\mid\mathcal F_\mu}(t,dx)\Big)\,d\lambda(t)
=\int_T I_n\big(t,f_m(t)\big)\,d\lambda(t)=0.
\]
Hence $0>\Phi_n(\mu^{f\mid\mathcal F_\mu})-\varepsilon$, and letting $\varepsilon\downarrow0$ gives
\[
\Phi_n(\mu^{f\mid\mathcal F_\mu})=0.
\]
Applying Lemma~\ref{lem:integration} once more,
\[
\int_T I_n\big(t,f(t)\big)\,d\lambda(t)=0,
\]
and therefore $f(t)\in\Gamma_n(t)$ for $\lambda$–a.e.\ $t$. As $n$ was arbitrary, $f$ is a selection of $\bigcap_{n\ge1}\Gamma_n$. Moreover, for each $t$ the map $z\mapsto G(t,z)$ has a closed graph, hence
\[
\bigcap_{n\ge1}\Gamma_n(t)\subset G(t,z_0)\ \text{for all } t \in T,
\]
and thus $f$ is a selection of $G_{z_0}$.

We next show that $\mu^{f\mid\mathcal G}$ is a cluster point of $\{\mu^{f_n\mid\mathcal G}\}$. Fix any $n \in \mathbf{N}$ and let $U$ be an arbitrary neighborhood of $\mu^{f\mid\mathcal G}$. Choose finitely many $\mathcal G$–measurable bounded Carathéodory functions
\(
\varphi_1,\dots,\varphi_k
\)
and $\varepsilon>0$ such that the neighborhood of $\mu^{f\mid\mathcal G},$
\[
V:=\Big\{\nu \in \mathcal R^\mathcal G(X):\ \big|\textstyle\int_T\!\int_X \varphi_j(t,x)\,\nu(t,dx)\,d\lambda
-\int_T\int_X \varphi_j\big(t,x\big)\mu^{f\mid\mathcal G}(t,dx)\,d\lambda\big|<\varepsilon,\ \forall j\le k\Big\}
\]
satisfies $V\subset U$.
Since $\mathcal G\subset\mathcal F$, each $\varphi_j$ is also $\mathcal F$–measurable. Using the same functions $\{\varphi_j\}$, define a neighborhood of $\mu^{f\mid\mathcal F}$ by
\[
V^{\mathcal F}:=\Big\{\nu \in \mathcal R^\mathcal F(X):\ \big|\textstyle\int_T\!\int_X \varphi_j(t,x)\,\nu(t,dx)\,d\lambda
-\int_T\!\int_X \varphi_j(t,x)\,\mu^{f\mid\mathcal F}(t,dx)\,d\lambda\big|<\varepsilon,\ \forall j\le k\Big\}.
\]
Because $\mu^{f\mid\mathcal F}$ is a cluster point of $\{\mu^{f_m\mid\mathcal F}\}$, there exists $m\ge n$ with $\mu^{f_m\mid\mathcal F}\in V^{\mathcal F}$. Applying Lemma~\ref{lem:integration} to each $\varphi_j$, both for $\mathcal F$ and for $\mathcal G$, we obtain
\[
\int_T\!\int_X \varphi_j(t,x)\,\mu^{f\mid\mathcal F}(t,dx)\,d\lambda
=\int_T \varphi_j\big(t,f(t)\big)\,d\lambda
=\int_T\!\int_X \varphi_j(t,x)\,\mu^{f\mid\mathcal G}(t,dx)\,d\lambda,
\]
and similarly
\[
\int_T\!\int_X \varphi_j(t,x)\,\mu^{f_m\mid\mathcal F}(t,dx)\,d\lambda
=\int_T \varphi_j\big(t,f_m(t)\big)\,d\lambda
=\int_T\!\int_X \varphi_j(t,x)\,\mu^{f_m\mid\mathcal G}(t,dx)\,d\lambda.
\]
Therefore, for all $j=1,2, \dots ,k$,
\[
\Big|\textstyle\int_T\!\int_X \varphi_j\,d\mu^{f_m\mid\mathcal G}\,d\lambda
-\int_T\!\int_X \varphi_j\,d\mu^{f\mid\mathcal G}\,d\lambda\Big|
=
\Big|\textstyle\int_T\!\int_X \varphi_j\,d\mu^{f_m\mid\mathcal F}\,d\lambda
-\int_T\!\int_X \varphi_j\,d\mu^{f\mid\mathcal F}\,d\lambda\Big|
<\varepsilon,
\]
so $\mu^{f_m\mid\mathcal G}\in V \subset U$. As $n$ and $U$ were arbitrary, $\mu^{f\mid\mathcal G}$ is a cluster point of $\{\mu^{f_n\mid\mathcal G}\}$.

Finally, by construction each $\mu^{f_n\mid\mathcal G}$ lies in the closed set $U^c$, so any of their cluster points must also lie in $U^c$. Yet we have shown $\mu^{f\mid\mathcal G}\in H(z_0)\subset U$ and that $\mu^{f\mid\mathcal G}$ is a cluster point of $\{\mu^{f_n\mid\mathcal G}\}$, a contradiction. Therefore $H$ is upper semicontinuous at $z_0$.
\end{proof}

\subsection{Proof of the necessity part of Theorem~\ref{the:main} }

Decompose \(T\) into two \(\mathcal F\)-measurable sets \(T_1,T_2\) with \(T=T_1\cup T_2\) and \(T_1\cap T_2=\emptyset \) such that the trace probability spaces \((T_1,\mathcal T^{T_1},\lambda^{T_1})\) and \((T_2,\mathcal T^{T_2},\lambda^{T_2})\) are, respectively, atomless and purely atomic; without loss of generality, assume \(\lambda(T_1)>0\).
Let \(n\ge2\) and pick pairwise distinct points \(x_0,x_1,\dots,x_n\) in the infinite Polish space \(X\).

\begin{proof}[Proof of B1 $\Rightarrow$ NE]
Define the $\mathcal F$-measurable closed-valued correspondence $F:T\rightrightarrows X$ by
\[
F(t)=
\begin{cases}
\{x_1,\dots,x_n\}, & t\in T_1,\\
\{x_0\}, & t\in T_2.
\end{cases}
\]
For each $k=1,\dots,n$, define an $\mathcal F$-measurable transition probability $G_k$ by
\[
G_k(t)=
\begin{cases}
\delta_{x_k}, & t\in T_1,\\
\delta_{x_0}, & t\in T_2,
\end{cases}
\]
where $\delta_x$ denotes the Dirac measure at $x$.
Let $f_k:T\to X$ be the $\mathcal F$-measurable selection of $F,$ defined as
\[
f_k(t)=
\begin{cases}
x_k, & t\in T_1,\\
x_0, & t\in T_2.
\end{cases}
\]
By construction, the RCD of $f_k$ given $\mathcal F$ satisfies
\[
\mu^{f_k\mid\mathcal F}=G_k,
\]
hence $G_k\in \mathcal R^{(\mathcal T,\mathcal F)}_{F}$ for each $k$.
By the convexity of $\mathcal R^{(\mathcal T,\mathcal F)}_{F}$,
we have
\[
H:=\frac{1}{n}\sum_{k=1}^n G_k \in \mathcal R^{(\mathcal T,\mathcal F)}_{F}.
\]
Therefore there exists a $\mathcal T$-measurable selection $f:T\to X$ of $F$ such that
\[
\mu^{f\mid\mathcal F}=H.
\]

Consider the restricted probability space $(T_1,\mathcal T^{T_1},\lambda^{T_1})$ and its sub-$\sigma$-algebra $\mathcal F^{T_1}$.
For $k=1,\dots,n,$ set
\[
E_k:=\{t\in T_1:\ f(t)=x_k\}\in\mathcal T^{T_1}.
\]
By the defining property of the RCD, for any $B\in\mathcal F^{T_1}$ we have
\[
\begin{aligned}
\lambda^{T_1}(E_k\cap B)
&=\frac{1}{\lambda(T_1)}\,\lambda(\{f = x_k\}\cap B)
=\frac{1}{\lambda(T_1)}\int_B \mu^{f\mid\mathcal F}(\{x_k\})\,d\lambda\\
&=\frac{1}{\lambda(T_1)}\int_B \frac{1}{n}\,d\lambda
=\frac{1}{n}\,\lambda^{T_1}(B).
\end{aligned}
\]
In particular, taking $B=T_1$ yields $\lambda^{T_1}(E_k)=1/n$, and hence for all $B\in\mathcal F^{T_1}$,
\[
\lambda^{T_1}(E_k\cap B)=\lambda^{T_1}(E_k)\,\lambda^{T_1}(B).
\]
Therefore each $E_k$ is independent of $\mathcal F^{T_1}$ under the probability space $(T_1,\mathcal T^{T_1},\lambda^{T_1})$. Since, for any $k$, we have $\lambda^{T_1}(E_k)=1/n$, it follows from Lemma 7 in \cite{He2017-cr} that $\mathcal T$ is nowhere equivalent to $\mathcal F.$
\end{proof}

\begin{proof}[Proof of B2 $\Rightarrow$ NE]
Let \(F\), \(\{G_k\}_{k=1}^n\), and \(H\) be exactly as introduced in the proof of B1 $\Rightarrow$ NE.
Let \(U\) be an arbitrary weak-topology neighborhood of \(H\).
Then, there exist bounded \(\mathcal F\)-measurable Carathéodory functions
\(f_1,\ldots,f_m:T\times X\to\mathbb R\) and \(\varepsilon>0\) such that
\[
V
:=\Bigl\{\mu:\ \bigl|\Phi_{f_j}(\mu)-\Phi_{f_j}(H)\bigr|<\varepsilon,\ j=1,\ldots,m\Bigr\}
\subset U,
\]
where 
\[\Phi_f(\mu):=\int_T\!\int_X f(t,x)\,\mu(t,dx)\,d\lambda(t).
\]

Consider these functions $\{f_j\}$ on \(T_1\times X\).
Since \(T_1\) is nonatomic and, for each \(j=1,\ldots,m\),
\[
\int_{T_1}\!\int_X f_j(t,x)\,H(t,dx)\,d\lambda(t)
=\int_{T_1} \frac1n\sum_{k=1}^n f_j(t,x_k)\,d\lambda(t),
\]
we can apply (extended) Lyapunov’s convexity theorem \cite[Theorem~IV.17]{Castaing1977-pk} to the \(\mathcal F\)-measurable functions \(F_k:T_1\to\mathbb R^m\) defined by
\[
F_k(t):=\bigl(f_1(t,x_k),\ldots,f_m(t,x_k)\bigr),\qquad k=1,\ldots,n,
\]
with equal weights \(\alpha_k=1/n\).
Then there exists an \(\mathcal F\)-measurable partition \(B_1,\ldots,B_n\) of \(T_1\) such that
\[
\int_{T_1} \frac1n\sum_{k=1}^n F_k\,d\lambda=\sum_{k=1}^n \int_{B_k} F_k\,d\lambda,
\]
equivalently, for each \(j=1,\ldots,m\),
\[
\int_{T_1} \frac1n\sum_{k=1}^n f_j(t,x_k)\,d\lambda(t)
=\sum_{k=1}^n \int_{B_k} f_j(t,x_k)\,d\lambda(t).
\]

Define an \(\mathcal F\)-measurable function \(s_{T_1}:T_1\to\{x_1,\ldots,x_n\}\) by \(s_{T_1}(t)=x_k\) on \(B_k\),
and extend it to \(T\) by
\[
s(t):=\begin{cases}
x_0,& t\in T_2,\\
s_{T_1}(t),& t\in T_1.
\end{cases}
\]
Then \(s\) is a selection of \(F\).
For every \(j=1,\ldots,m\),
\[
\int_T f_j\!\bigl(t,s(t)\bigr)\,d\lambda(t)
=\int_{T_2} f_j(t,x_0)\,d\lambda(t)+\sum_{k=1}^n\int_{B_k} f_j(t,x_k)\,d\lambda(t)
=\int_T\!\int_X f_j(t,x)\,H(t,dx)\,d\lambda(t).
\]
Consequently \(\delta_s\in V\subset U\), where \(\delta_s(t,B) = \delta_{s(t)}(B)\) for all $t \in T$ and $B \in \mathcal{B}(X).$

Since \(s\) is \(\mathcal F\)-measurable, its RCD given \(\mathcal F\) is \(\delta_s\), i.e.,
\( \mu^{s\mid\mathcal F}=\delta_s\).
Hence \(\mu^{s\mid\mathcal F}\in U\).
Because \(U\) was arbitrary, \(H\) belongs to the weak closure of \(\mathcal R_{F}^{(T,\mathcal F)}\).
By assumption, \(\mathcal R_{F}^{(T,\mathcal F)}\) is closed; hence \(H\in\mathcal R_{F}^{(T,\mathcal F)}\).
The remainder is identical to the proof of B1 $\Rightarrow$ NE.
\end{proof}

\begin{proof}[Proof of B3 $\Rightarrow$ NE, B4 $\Rightarrow$ NE, and B5 $\Rightarrow$ NE]
First suppose either B3 or B4. Then, for every compact-valued $\mathcal F$-measurable correspondence $F:T\rightrightarrows X$, the set $\mathcal R_F^{(\mathcal T,\mathcal F)}$ is closed. Since the correspondence $F$ used in the proof of $\mathrm{B1}\Rightarrow\mathrm{NE}$ is compact-valued, the proof of $\mathrm{B2}\Rightarrow\mathrm{NE}$ applies without modification, and nowhere equivalence follows.

Next assume B5. For the $H$ defined in the proof of $\mathrm{B1}\Rightarrow\mathrm{NE}$ there exists a $\mathcal T$-measurable function $f:T\to X$ such that $\mu^{f\mid\mathcal F}=H$. The remainder of the argument is the same as in the proof of $\mathrm{B1}\Rightarrow\mathrm{NE}$.
\end{proof}

\section*{Acknowledgment}
This study is supported by JST SPRING, Grant Number JPMJSP2124.11.

\bibliographystyle{elsarticle-harv} 
\bibliography{paperpile-2}

\end{document}